\documentclass{article}
\usepackage{times}
\usepackage{amsmath}
\usepackage{amssymb}
\usepackage{amsthm}
\usepackage[english]{babel}
\usepackage{times}
\usepackage[T1]{fontenc}
\usepackage[latin1]{inputenc}
\usepackage[all]{xy}
\usepackage{color}
\newtheorem{teo}{Theorem}[section]
\newtheorem{coro}[teo]{Corollary}
\newtheorem{defn}[teo]{Definition}
\newtheorem{lem}[teo]{Lemma}
\newtheorem{pro}[teo]{Proposition}
\newtheorem{prob}[teo]{Problem}

\newcommand{\B}{\mathbb{B}}
\newcommand{\HH}{\mathbb{H}}
\newcommand{\s}{\mathbb{S}}
\newcommand{\rr}{\mathbb{R}}

\DeclareMathOperator{\IIm}{Im}
\DeclareMathOperator{\RRe}{Re}
\DeclareMathOperator{\ext}{ext}

\usepackage[]{graphicx}
\usepackage{geometry}
\title{ \bf The Bohr Theorem for slice regular functions}
\date{}

\author{Chiara Della Rocchetta \footnote{The three authors acknowledge the support of G.N.S.A.G.A. of INdAM and MIUR (Research Project ``Propriet\`a geometriche delle variet\`a reali e complesse'')}\\
\normalsize Dipartimento di Matematica ``U. Dini'', Universit\`a di Firenze \\
\normalsize Viale Morgagni 67/A, 50134 Firenze, Italy, mchidel@gmail.com \\
\and Graziano Gentili $^*$ \\
\normalsize Dipartimento di Matematica ``U. Dini'', Universit\`a di Firenze \\
\normalsize Viale Morgagni 67/A, 50134 Firenze, Italy,  gentili@math.unifi.it \\
\and Giulia Sarfatti $^*$\\
\normalsize Dipartimento di Matematica ``U. Dini'', Universit\`a di Firenze \\
\normalsize Viale Morgagni 67/A, 50134 Firenze, Italy,  sarfatti@math.unifi.it \\
}

\begin{document}

\maketitle                   

\begin{abstract}
In this paper we prove the Bohr Theorem for slice regular functions. Following the historical path that led to the proof of the classical Bohr Theorem, we also extend the Borel-Carath\'eodory Theorem to the new setting.
\end{abstract}

\vskip 0.5 cm
 
{\bf keywords:} Functions of a quaternionic variable, Bohr Theorem


\section{Preface} 

The search for quaternionic analogs of the class of complex holomorphic functions has produced, in approximately a Century, a variety of definitions having different approaches. These definitions accompany the most famous and successful one due to Fueter, \cite{fueter1, fueter2}, that has produced the well recognized theory of Fueter regular functions (see, e.g., the nice survey \cite{Sudbery}, the more recent \cite{libro daniele}, and references therein). There is a recent definition inspired by Cullen, \cite{Cullen}, and given in its full generality by Gentili and Struppa in \cite{G.S., GSAdvances}, that has the advantage - when compared with Fueter's definition - to include in the class of regular functions the natural polynomials and power series of the form $\sum_{n=0}^\infty q^na_n$, with $a_n$ belonging to the skew field $\mathbb{H}$ of quaternions. This definition, presented in detail in the Preliminaries, originated the theory of slice regular functions, that is already well established and fast developing also in the more general setting of Clifford Algebras, where it produced interesting and deep applications (see e.g. \cite{monogenic, libro2}),   Slice regular functions have properties that are typical, and in some sense characterizing, of the class of complex holomorphic functions. Just to mention a few of the basic results that hold for slice regular functions (sometimes in a peculiarly different way with respect to complex  holomorphic functions), we recall the power and Laurent series expansion, the Cauchy and Pompeiu Representation Formulas, the Cauchy estimates, the Maximum (and Minimum) Modulus principle, the Identity Principle, and the Open Mapping Theorem (\cite{kernel, pompeiu, ext, open, zeri, power, CaterinaCV, singularities}).

Holomorphic functions of one complex variable have certain deep geometric properties, that greatly contribute to the beauty of their theory. Some of these properties concern the geometry of the image, or inverse image, of the unit ball of $\mathbb{C}$ through holomorphic functions. One of the classical results in this setting is due to Bohr, who, while studying questions of Diophantine approximation, encountered the following:
\begin{prob}\label{Pb}
Let $x\in (0,1)$ be a real number. Establish whether it is possible to find a power series $\sum_{n= 0}^\infty  a_nz^n$, with $a_n \in \mathbb{C}$, such that:
\begin{enumerate}
\item $f(z)=\sum_{n= 0}^\infty  a_nz^n$ is holomorphic for $|z|<1$ and continuous for $|z|\le1$;
\item $|f(z)|<1$ for $|z|\le1$;
\item $\sum_{n= 0}^\infty  x^n|a_n|>1$.
\end{enumerate}
\end{prob}
\noindent Bohr himself presented in \cite{Bohr} first a partial solution of Problem \ref{Pb} and then the complete result, known as Bohr Theorem, relating the proof due to Wiener.

\begin{teo}[Bohr, complex case]
Let 
$$f(z)=\sum_{n= 0}^\infty  a_nz^n$$ be holomorphic for $|z|<1$, continuous for $|z|\le1$, and let $|f(z)|<1$ for all $|z|\le 1$. Then
$$
\sum_{n= 0}^\infty  |a_nz^n| <1
$$
for $|z|\le \frac{1}{3}$.
Moreover $\frac{1}{3}$ is the largest radius for which the statement is true.
\end{teo}

For different reasons, recent studies of various authors are dedicated to generalize the Bohr Theorem to new settings, for instance to the case of holomorphic functions of several complex variables (see e.g. \cite{ A1, A2, Boas, Djakov, P}). In \cite{Kglu}, weighted Laplace-Beltrami operators associated with the hyperbolic metric of the unit ball in $\mathbb{C}^n$ are considered:  the author studies a Bohr phenomenon on the spaces of their solutions.The case of classical monogenic functions is treated in \cite{GM, GM1, GM2, M}.

In this paper we prove the Bohr Theorem for slice regular functions of one quaternionic variable. This result contributes to the construction of a geometric theory of slice regular functions, that reveals to posses a richness similar to that of holomorphic functions. We follow the historical approach to the Bohr Theorem, and first of all we prove the analog of the celebrated Borel-Carath\'eodory Theorem, that certainly has its own independent interest. Due to the peculiarities of the quaternionic setting, this proof requires a certain technical effort. We then establish the weak version of the Bohr theorem, and then proceed to the proof of its sharp version following the guidelines due to Wiener and used by Bohr in \cite{Bohr}. Recent new approaches to the study of quaternionic rational transformations (\cite{cinzia}) and new results on regular quaternionic rational transformations (\cite{regular}) are heavily used to obtain our results. We believe that these achievements may also help to open new perspectives in the investigation concerning quaternionic regular Dirichlet series.

The paper is organized as follows. After Section 2, dedicated to the necessary preliminaries on slice regular functions, in Section 3 we prove the Borel-Carath\'eodory Theorem and a weak version of the Bohr Theorem. The sharp version of the Bohr Theorem is proved in Section 4.

\section{Preliminaries}

Let $\HH$ be the skew field of quaternions. Each element $q$ of $\HH$ is of the form $q= x_0 +x_1i+x_2j+x_3k$, where the $x_l$ are real numbers and $i,j,k$ satisfy the relations
$$i^2=j^2=k^2=-1,\; ij=k=-ji,\; jk=i=-kj,\; ki=j=-ik.$$
The {\em real} and {\em imaginary} part of a quaternion $q$ are defined as $\RRe(q)=x_0$ and $\IIm(q)=x_1i +x_2j +x_3k$, its {\em conjugate} as $\bar{q}=\RRe(q)-\IIm(q)$ and its {\em modulus} by $|q|^2=q\bar{q}=\RRe(q)^2+|\IIm(q)|^2$. The multiplicative inverse of each $q\neq 0$ is then $q^{-1}=\frac{\bar{q}}{|q|^2}$.
Notice that for all $q \in \HH$ that is not real, $\frac{\IIm(q)}{|\IIm(q)|}$ defines an imaginary unit, i.e. a quaternion whose square equals $-1$. Then every $q\in \HH$ can be written as $q=x+yI$ where $x,y \in \rr$ and $I$ is an element of the unit $2$-sphere of purely imaginary quaternions,
$$\mathbb{S}=\{q\in \HH \ | \ q^{2}=-1 \}.$$
For every $I\in \mathbb{S}$ we will denote by $L_I$ the plane $\mathbb{R}+I\mathbb{R}$, isomorphic to $\mathbb{C}$, and, if $\Omega \subset \HH$, by $\Omega_I$ the intersection $\Omega \cap L_I$.

Let us now recall the definition of slice regularity
\begin{defn}
Let $\Omega$ be a domain in $\HH$. A function $f: \Omega \rightarrow \HH$ is called {\em slice regular} if, for all $I\in \mathbb{S}$,  its restriction $f_I$ to $\Omega_I$ has continuous partial derivatives and satisfies
$$\overline{\partial_I} f(x+yI):=\frac{1}{2}\Big(\frac{\partial}{\partial x}  +I\frac{\partial}{\partial y}\Big)f_I(x+yI)=0$$
for every $x+yI \in \Omega_I$.
\end{defn}
\noindent In the sequel we may refer to the vanishing of $\overline{\partial_I}f$ saying that the restriction $f_I$ is holomorphic on $\Omega_I$.
In what follows, for the sake of simplicity, we will omit the prefix slice when referring to slice regular functions.
Even if the definition of regularity can be given for any domain in $\HH$, to avoid degenerate cases, like regular functions that are not even continuous, we need a special class of domains, introduced in \cite{ext, CaterinaCV}.
\begin{defn}
Let $\Omega$ be a domain in $\HH$. We say that $\Omega$
is a {\em slice domain} if \ $\Omega \cap \mathbb{R}\neq \varnothing$ and
if, for every $I \in \mathbb{S}$, $\Omega_I$ is a domain in $L_I$.
\end{defn}
\noindent Another condition that is natural to require for a domain of definition of a regular function, and that appears in \cite{kernel, ext, open}, is the following.
\begin{defn}
A subset $U$ of $\HH$ is said to be {\em axially symmetric} if for all $x+yI\in U$, with $x,y \in \rr, \ y \neq 0,$ and $I \in \s$, the whole $2$-sphere $x+y\s=\{x+yJ \ | \ J \in \s \}$ is contained in $U$.  
\end{defn}
\noindent We point out that axially symmetric sets were previously introduced in \cite{Cullen}, under the name of intrinsic domains. For the sake of shortness we will refer to axially symmetric sets, simply as symmetric sets.
The symmetric slice domains play the role that the domains of holomorphy play in classical complex analysis.
In \cite{ext} it is indeed proved that every regular function, defined on a slice domain $\Omega$, can be uniquely extended to the smallest symmetric domain containing $\Omega$.
 
\noindent A natural notion of derivative can be given for regular functions, as follows (see \cite{GSAdvances})
\begin{defn}
Let $\Omega$ be a slice domain in $\HH$, and let $f: \Omega \rightarrow \HH$ be a regular function. The {\em slice derivative} of $f$ at $q=x+yI$ is defined as
\begin{equation*}
\partial_S f(x+yI)= \frac{\partial}{\partial x} f(x+yI).
\end{equation*}
\end{defn}
\noindent Notice that this definition is well posed because it is applied only to regular functions.
Notice also that the operators $\partial_S$ and $\overline{\partial_I}$ commute, 
therefore the slice derivative of a regular function is still regular, and we can iterate the differentiation, \cite{GSAdvances},
$$\partial_S^n f= \frac{\partial^n}{\partial x^n}f \quad \text{for any} \quad n \in \mathbb{N}.$$

As stated in \cite{GSAdvances}, a quaternionic power series $\sum_{n\geq 0}q^na_n$ with $\{a_n\}_{n\in \mathbb{N}}\subset \HH$ defines a regular function in its domain of convergence, which proves to be a ball $B(0,R)=\{q\in \HH \,|\, |q|<R\}$ with $R$ equal to the radius of convergence of the power series. Moreover, in \cite{GSAdvances}, it is also proved that
\begin{teo}
A function $f$ is regular on $B=B(0,R)$ if and only if $f$ has a power series expansion
$$f(q)=\sum_{n \geq 0}q^na_n \quad\text{with} \quad a_n=\frac{1}{n!}\frac{\partial^n f}{\partial x^n}(0).$$
\end{teo}  
\noindent From a slicewise version of the Cauchy Integral Formula, see \cite{GSAdvances}, it follows that the $n$-th coefficient $a_n$ of the power series expansion of a regular function $f:B(0,R) \rightarrow \HH$ has an integral representation. More precisely, for all $n \geq 0$, if $I \in \s$, $r \in (0,R)$ and $\Delta_I(0,r)=\{ z \in L_I \ | \ |z|<r\} $, then
\begin{equation}\label{intrep} 
a_n= \frac{1}{2\pi I}\int_{\partial \Delta_I(0,r)} \frac{dz}{z^{n+1}}f(z).
\end{equation}
One of the key tools of the theory of regular functions, that connects slice regularity and classical holomorphy, is the following, \cite{GSAdvances}:
\begin{lem}[Splitting Lemma]\label{split}
Let $\Omega$ be a slice domain in $\HH$. If $f$ is a regular function on $\Omega$, then for every $I \in \mathbb{S}$ and for every $J \in \mathbb{S}$, $J$ orthogonal to $I$, there exist two holomorphic functions $F,G:\Omega_I \rightarrow L_I$, such that for every $z=x+yI \in \Omega_I$, the following equality holds
$$f_I(z)=F(z)+G(z)J.$$
\end{lem}
\noindent One of the first consequences of the previous result is the following version of the Identity Principle, \cite{GSAdvances}:
\begin{teo}[Identity Principle]
Let $f$ be a regular function on a slice domain $\Omega$. Denote by $Z_f$ the zero set of $f$, $Z_f=\{ q \in \Omega \  | \, f(q)=0 \}$. If there exists $I \in \mathbb{S}$ such that $\Omega_I \cap  Z_f$ has an accumulation point in $\Omega_I$, then
$f$ vanishes identically on $\Omega$.
\end{teo}

\noindent The pointwise product of two regular functions is not, in general, regular. To maintain the regularity, a new multiplication operation, the $*$-product, was introduced. 
 On open balls centered at the origin, the $*$-product of two regular functions can be defined by means of their power series expansions, \cite{zeri}, extending the classical $*$-product of polynomials  with coefficients in a non commutative ring (see, e.g., \cite{Lam}). As we will see, the generalization to the symmetric slice domains is based on the following result, whose  proof is in \cite{kernel, ext}.
 \begin{lem}[Extension Lemma]\label{extensionlemma}
 Let $\Omega$ be a symmetric slice domain and choose $I \in \s$. If $f_I: \Omega_I \rightarrow \HH$ is holomorphic, then setting 
\begin{equation*}
f(x+yJ)=\frac{1}{2}[f_I(x+yI)+f_I(x-yI)] +J\frac{I}{2}[f_I(x-yI)-f_I(x+yI)] 
\end{equation*} 
extends $f_I$ to a regular function $f: \Omega \rightarrow \HH$. Moreover $f$ is the unique extension and it is denoted by $\ext(f_I)$.
\end{lem}
\noindent In order to define the regular product of $f$ and $g$, regular functions on a symmetric slice domain $\Omega$, take $I,J \in \s$, with $I$ orthogonal to $J$, and choose holomorphic functions $F,G,H,K: \Omega_I \rightarrow L_I$ such that for all $z\in \Omega_I$
$$f_I(z)=F(z)+G(z)J \quad \text{and} \quad g_I(z)=H(z)+K(z)J.$$
Let $f_I * g_I : \Omega_I \rightarrow L_I$ be the holomorphic function defined as
$$f_I * g_I(z)=[F(z)H(z)-G(z)\overline{K(\overline{z})}]+[F(z)K(z)+G(z)\overline{H(\overline{z})}]J.$$
The following definition is given in \cite{ext}.
\begin{defn}
Let $\Omega$ be a symmetric slice domain in $\HH$, and let $f,g: \Omega \rightarrow \HH$ be regular functions.
The {\em regular product} (or {\em $*$-product}) of $f$ and $g$ is the function defined as
$$f*g(q)=\ext(f_I * g_I)(q),$$
regular on $\Omega$.
\end{defn}
\noindent Notice that the $*$-product is associative and is not, in general, commutative. Its relation with the pointwise product is clarified by the following result (see \cite{zeri, ext}).
\begin{pro}\label{trasf}
Let $f$ and $g$ be regular functions on a symmetric slice domain $\Omega$. Then, for all $q\in \Omega$,
\begin{equation}\label{prodstar}
f*g(q)= \left\{ \begin{array}{ll}
 f(q)g(f(q)^{-1}qf(q)) & \text{if} \quad f(q)\neq 0\\
0 & \text{if} \quad f(q)=0
\end{array}
\right.
\end{equation}
\end{pro}
\noindent 
Notice that if $q=x+yI$ (and if $f(q)\neq 0$), then $f(q)^{-1}qf(q)$ has the same modulus and same real part as $q$, hence $f(q)^{-1}qf(q)$ lies in the same $2$-sphere $x+y\mathbb{S}$ as $q$. We have that a zero $x_0+y_0I$ of the function $g$ is not necessarily a zero of $f*g$, but an element on the same sphere $x_0+y_0\mathbb{S}$ is.
To present a characterization of the structure of the zero set of a regular function $f$ we need to introduce the following functions.

\begin{defn}
Let $f$ be a regular function on a symmetric slice domain $\Omega$ and suppose that for $z\in \Omega_I$, the splitting of $f$ with respect to $J$ is $f_I(z)=F(z)+G(z)J$. Consider the holomorphic function
$$f_I^c(z)=\overline{F(\overline{z})}-G(z)J.$$
The {\em regular conjugate} of $f$ is the function defined  by  
$$f^c(q)=\ext(f_I^c)(q).$$
The {\em symmetrization} of $f$ is the function defined by
$$f^s(q)=f*f^c(q)=f^c*f(q).$$
\noindent Both $f^c$ and $f^s$ are regular functions on $\Omega$.
\end{defn}
\noindent In \cite{ext} it has been proved that the function $f^s$ is slice preserving, i.e. $f^s(L_I)\subset L_I$ for every $I \in \s$. 
Thanks to this property it is possible to prove (see for instance \cite{zeri}) the following 
\begin{teo}
Let $f$ be a regular function on a symmetric slice domain $\Omega$. If $f$ does not vanish identically, then its zero set consists of the union of isolated points and isolated $2$-spheres of the form $x +y \mathbb{S}$ with $x,y \in \mathbb{R}$, $y \neq 0$.
\end{teo} 
\noindent We are now able to define the inverse element of a regular function $f$ with respect to the $*$-product.
Recall that  $Z_{f^s}$denotes  the zero set of the symmetrization $f^s$ of $f$.
\begin{defn}
Let $f$ be a regular function on a symmetric slice domain $\Omega$. If $f$ does not vanish identically, its {\em regular reciprocal} is defined as the function
$$f^{-*}(q):=f^s(q)^{-1}f^c(q)$$
regular on $\Omega \setminus Z_{f^s}$.
\end{defn}

It is then possible to consider regular quotients of the form $f^{-*}*g$, defined outside the zero set of $f$, that are related to pointwise quotients by the following result, \cite{CaterinaCV}.

\begin{pro}\label{Caterina} 
Let $\Omega$ be a symmetric slice domain in $\HH$ and let $f$ and $g$ be regular functions on $\Omega$. If \  $T_f : \Omega \setminus Z_{f^s} \rightarrow  \Omega \setminus Z_{f^s}$ is defined as
$$T_{f}(q)=f^c(q)^{-1}qf^c(q),$$ then 
$$f^{-*}*g(q)=f(T_{f}(q))^{-1}g(T_{f}(q)) \quad \text{for every} \quad q \in B \setminus Z_{f^s}.$$ 
Furthermore, $T_f$ and $T_{f^c}$ are mutual inverses so that $T_f$ is a diffeomorphism.
\end{pro}

A very important result in analogy with the complex case is the following (see \cite{open2}).
\begin{teo}[Maximum Modulus Principle]\label{PMM}
Let $f: \Omega \rightarrow \HH$ be a regular function on a slice domain $\Omega$. If $|f|$ has a relative maximum in $\Omega$, then $f$ is constant in $\Omega$. 
\end{teo}
Among all the consequences that the previous result yields, there is also the analog of the Open Mapping Theorem for regular functions, whose statement needs a preliminary definition.
\begin{defn}
 Let $\Omega$ be a symmetric slice domain and let $f : \Omega \rightarrow \HH$ be a regular function. We define the {\em degenerate set} of $f$ as the union \ $D_f$ of the $2$-spheres $x+y\s$
(with $y \neq 0$) such that $f|_{x+y\s}$ is constant.
\end{defn}
\noindent We are now able to recall the following result (see \cite{open, open2}) and one of its consequences.
\begin{teo}[Open Mapping Theorem]
Let $\Omega$ be a symmetric slice domain and let $f: \Omega \rightarrow \HH$ be a regular function. 
If $D_f$ is the degenerate set of $f$, then $f: \Omega \setminus D_f \rightarrow \HH$ is open.
\end{teo}
\begin{coro} \label{openmapping}
Let $f$ be a regular function on $B=B(0,R)$. Then $f(B)$ is an open set.
\end{coro}

We conclude this section, devoted to the preliminary results, with a mention to the fractional linear transformations of the space of the quaternions, which will turn out to be remarkably useful in our proofs. 
Let $\hat{\HH}$ be the Alexandroff compactification of $\HH$.
\begin{defn}
A function $L:\hat{\HH}\rightarrow \hat{\HH}$ is a {\em fractional linear transformation} if it is of the form 
$$L(q)=(qa+b)^{-1}(qc+d)$$ 
where $a,b,c,d \in \HH$ are such that $A= \left[ \begin{array}{ll}
             a & b \\
             c & d \\
             \end{array} \right]$ is invertible, i.e. such that the Dieudonn\'e determinant of $A$, $$det_{\mathbb{H}}(A)= \sqrt{|a|^2|d|^2 + |c|^2|b|^2
-2\RRe(c\overline{a}b\overline{d})}$$ 
does not vanish. 
\end{defn}
\noindent Even if they are not regular, the fractional linear transformations are homeomorphism of $\hat{\HH}$ onto itself and have some nice properties. 
For instance, let $\mathcal{F}_3$ be the family of $3$-dimensional spheres and $3$-dimensional affine subspaces of $\HH$. 
\begin{pro}
Every fractional linear transformation $L$, sends the family $\mathcal{F}_3$ into itself.
\end{pro}
For details on this nice result see \cite{cinzia}.

\section{The Borel-Carath\'eodory Theorem and a weak Bohr Theorem}
Our first purpose is to prove an analog of the Borel-Carath\'eodory Theorem for regular functions. This result will show how we can bound the modulus of a regular function with the modulus of its real part. In addition to its independent importance, we recall that, historically, the complex Borel-Carath\'eodory Theorem has been used as a fundamental tool to prove a weak version of the Bohr Theorem, see \cite{Bohr}.
\begin{teo}[Borel-Carath\'eodory]\label{Carath\'eodory}
Let $q_0 \in \mathbb{R}$, and let $ f: B=B(q_0,r) \rightarrow \mathbb{H}$ be a regular function on $\overline{B(q_0,r)}$. Set
$A= \max_{|q-q_0|=r} |\RRe f(q)| $, and
 $f(q_0)=\beta + \gamma I$ with $\gamma, \beta \in \mathbb{R}$ and $I \in \mathbb{S}$. If $\varrho \in
\mathbb{R}$ is such that $0<\varrho <r$, then 
$$ |f(q)| \leq |\gamma| + |\beta| \frac{r+ \varrho}{r- \varrho} +2A \frac{\varrho}{r- \varrho}$$
for all $q\in B$ such that $|q-q_0|\leq \varrho $. 
\end{teo}

\begin{proof}
If the function $f$ is constant, then the statement is trivially true.
If $f$ is not constant then also its real part $\RRe f(q)$ is not constant. Otherwise the image of $B$ under $f$ would be contained in a $3-$dimensional space, contradicting the statement of Corollary \ref{openmapping}.    
Let $J\in \s$ be such that $\RRe f(q)$ is not constant on $L_J$. By definition $\beta=\RRe f(q_0)\leq A$, and moreover equality cannot hold. In fact, suppose, that $\beta=A$.  By the Splitting Lemma \ref{split} there exist $F,G: B_J=B \cap L_J\rightarrow  L_J$ holomorphic functions, and $K\in \s$, $K$ orthogonal to $J$, such that
$$f(z)=F(z)+G(z)K$$
for every $z\in B_J$.
Then $\RRe F(q_0)=\RRe f(q_0)=A$. Since $F$ is holomorphic then $\RRe F(q)$ is harmonic  and hence, if it attains its maximum at an interior point, it must be constant (together with $F$).  As a consequence $\RRe f(q)$ would be constant on $L_J$, a contradiction. Hence $\beta < A$.

Set $w_0=f(q_0)-A$ and let $$h(q)=f(q)-A.$$ 
Consider the function 
\begin{equation}
 H(q)= ( f(q)-A + \overline{w_0} )^{-*}*( f(q)-A -w_0).
\end{equation}
The function $H$ is regular for $|q-q_0|\leq r $; indeed if $f(q)-A + \overline{w_0}=0$ for some $q\in B(q_0,r)$, then also the symmetrization  $(f(q)-A + \overline{w_0})^s$ vanishes somewhere in the same ball. In particular its real part vanishes, hence  
$$0= \RRe(f(q))-A+ \RRe(\overline{w_0}) \leq A-A+ \beta -A= \beta -A.$$
That implies $\beta \geq A$, a contradiction. 
By Proposition \ref{Caterina} we can express $H(q)$ in terms of the linear fractional transformation
$$g(q)=(q+\overline{w_0})^{-1}(q-w_0)$$ 
and of the transformation $$T(q)=((h(q)+\overline{w_0})^c)^{-1}q(h(q)+\overline{w_0})^c.$$
Namely
$$H(q)=g\circ h \circ T(q).$$
From this expression we easily get that $H(q_0)=0$. In fact $q_0 \in \rr$, hence $T(q_0)=q_0$ and $g(h(q_0))=g(w_0)=0$. Moreover $|H(q)|\leq 1$ in $B$. In fact the linear fractional transformation $g$ (sends the family $\mathcal{F}_3$ of $3$-spheres and affine $3$-subspaces in itself and)  maps the $3$-space $\{q \in \mathbb{H} |\RRe(q)=0\}$ onto the unit sphere $\mathbb{S}^3$. Since $H(q_0)=0$ and $\RRe(h(T(q))\leq0$ for all $q\in B$, we get that 
\begin{equation}\label{H0}
|H(q)|\leq 1 \quad \text{ for all} \quad q\in  B.
\end{equation} 
Furthermore, for all $q$ such that $0<|q-q_0|\leq \varrho <r$, the following inequality holds
\begin{equation}\label{H1}
|H(q)|\leq \frac{\varrho}{r}.
\end{equation}
In fact, consider the function defined as $(q-q_0)^{-*}* H(q)$.  Since $H(q_0)=0$ and $q_0\in \mathbb{R}$, this is a  regular function and $(q-q_0)^{-*}* H(q)=(q-q_0)^{-1}H(q).$   
Then by the Maximum Modulus Principle \ref{PMM} and by inequality \eqref{H0} we have
\begin{equation}\label{H2}
|(q-q_0)^{-*}*H(q)|=\frac{|H(q)|}{|q-q_0|}\leq\frac{1}{r}
\end{equation}
for all $q\in B$. By the Maximum Modulus Principle \ref{PMM} we have also that for all $q \in B(q_o,\varrho)$
 \begin{equation}
\frac{|H(q)|}{|q-q_0|}\leq \max_{|q-q_0|=\varrho}\frac{|H(q)|}{|q-q_0|}=\max_{|q-q_0|=\varrho}\frac{|H(q)|}{\varrho}.
\end{equation}
Since inequality \eqref{H2} holds for all $q \in B$ we get that in $B(q_o,\varrho)$
$$|H(q)|\leq \frac{\varrho}{r}.$$ 
Now, we want to use this inequality to estimate the modulus $|f|$.
Recall that $$H(q)=(f(T(q))-A+\overline{w_0}) ^{-1}(f(T(q))-A-w_0).$$
Notice that if $q+\overline{w_0}\neq0$, then
\begin{equation*}
\begin{aligned}
&(q+\overline{w_0})^{-1}(q-w_0)=(q-w_0)(q+\overline{w_0})^{-1}  \\
\text {if and only if}\quad &(q-w_0)(q+\overline{w_0})=(q+\overline{w_0})(q-w_0) \\
\text {if and only if}\quad &q^2 -w_0q +q\overline{w_0} -|w_0|^2=q^2 +\overline{w_0}q -qw_0 -|w_0|^2 \\
\text {if and only if}\quad &(-w_0 -\overline{w_0})q=q(-w_0 -\overline{w_0}).
\end{aligned}
\end{equation*}
Since $-w_0 -\overline{w_0}\in \mathbb R$ we get that the last equality holds.
Therefore we can write $$H(q)=(f(T(q))-A-w_0)(f(T(q))-A+\overline{w_0}) ^{-1}$$
and hence
$$H(q)(f(T(q))-A+\overline{w_0})=(f(T(q))-A-w_0)$$
that yields
$$(H(q)-1)f(T(q))=H(q)(A-\overline{w_0})-A-w_0.$$
Recalling the definition of $w_0$, we get
\begin{equation*}
\begin{aligned}
|f(T(q))|&=|(H(q)-1)^{-1}(H(q)(-\overline{f(q_0)}+2A)-f(q_0))|\\
&=|(H(q)-1)^{-1}[(H(q)-1)f(q_0)+H(q)(-\overline{f(q_0)}+2A-f(q_0))]|\\
&=|\beta+\gamma I+(H(q)-1)^{-1}H(q)(-\beta+\gamma I+2A-\beta-\gamma I)|\\
&=|\beta+\gamma I+2(H(q)-1)^{-1}H(q)( A-\beta)|.\\
\end{aligned}
\end{equation*}
Using the triangle inequality and inequality \eqref{H1}  
we have that for $0\leq |q-q_0|\leq \varrho \leq r$
\begin{equation}\label{car}
\begin{aligned}
|f(T(q))|&\leq |\beta|+|\gamma| +2\frac{(A+|\beta|)|H(q)|}{|H(q)-1|}\\
&\leq |\beta|+|\gamma| +2\frac{(A+|\beta|)\frac{\varrho}{r}}{1-\frac{\varrho}{r}}\\
&=|\gamma| + |\beta| (1+2\frac{\frac{\varrho}{r}}{1-\frac{\varrho}{r}})+2\frac{A\frac{\varrho}{r}}{1-\frac{\varrho}{r}}\\
&=|\gamma| + |\beta| \frac{r+\varrho}{r- \varrho}+A\frac{2\varrho}{r-\varrho}.
\end{aligned}
\end{equation}

\noindent Since $h(q)+\overline{w_0}\neq 0$ for all $q\in B$, the transformation $T=T_{h(q)+\overline{w_0}}$ is a diffeomorphism of $B$ onto itself. Hence inequality \eqref{car} implies that $$ |f(q)| \leq |\gamma| + |\beta| \frac{r+ \varrho}{r- \varrho} +2A \frac{\varrho}{r- \varrho}$$
for all $q$ such that $|q-q_0|\leq \varrho $. 
\end{proof}

To retrace the historical approach to the complex case, we begin by proving here also the analog of the weak version of the Bohr Theorem (see \cite{Bohr}). 
From now on,  $\B$ will denote the open unit ball of $\HH$, 
$$\B=\{q \in \HH \ | \ |q| < 1\}.$$
\begin{teo}[Bohr, weak version]\label{bohr1/6}
Let $f(q)=\sum_{n\geq 0} q^n a_n$ be a regular function on the unit ball $\B$, continuous on the closure $\overline{\B}$, such that $|f(q)|<1$ for all $|q|\leq 1$.
Then
\begin{equation*}
\sum_{n\geq 0}|q^n a_n|<1
\end{equation*}
for all $|q| \leq \frac{1}{6}$.
\end{teo}

\begin{proof}
Up to right-multiplying $f$ by the unitary constant $\frac{\overline{a_0}}{|a_0|}$, we can suppose $a_0 \in [0,1)$. Consider the function $g(q)=f(q)-a_0$, regular on $\B$, continuous on $\overline{\B}$.
Set $$A=\max_{|q|=1}\RRe g(q)$$ and $$m=\max_{|q|=\frac{1}{2}}|g(q)|.$$
By the Borel-Carath\'eodory Theorem \ref{Carath\'eodory} we get then that for all $|q|\leq \frac{1}{2}$ 
$$|g(q)|\leq 2A$$
and hence also that
$$m=\max_{|q|=\frac{1}{2}}|g(q)|\leq 2A.$$ 
Moreover, since $a_0$ and $A$ are non negative,
$$a_0+A=a_0+\max_{|q|=1}\RRe g(q)=\max_{|q|=1}\RRe (a_0+g(q))\leq \max_{|q|=1}|a_0+g(q)|=\max_{|q|=1}|f(q)|<1.$$
Therefore $A<1-a_0$ and $m<2(1-a_0)$.
Now we want to show that $|a_n|<2^{n+1}(1-a_0)$ for all $n\geq1$.
Let $J\in \s$, and set $$\Delta_J=\{z\in L_J :  |z|<\frac{1}{2}\}.$$
Consider the integral representation of the coefficients $a_n$, 
$$a_n=\frac{1}{2\pi J} \int_{\partial\Delta_J}\frac{dz}{z^{n+1}}g(z)$$
for all $n\geq 1$.
Hence
$$|a_n|\leq \frac{1}{2\pi} \int_{\partial\Delta_J}\frac{|g(z)|}{|z^{n+1}|}dz\leq m 2^{n}<2^{n+1}(1-a_0)$$
for all $n\geq 1$.
Therefore
\begin{equation*} 
\begin{aligned}
\sum_{n\geq 0}|q^n a_n|&=a_0 +\sum_{n\geq 1}|q^n a_n| < a_0 +\sum_{n\geq 1}|q|^n 2^{n+1}(1-a_0)\\
&=a_0 + 2(1-a_0)\sum_{n\geq 1}|q|^n 2^n=a_0 + \frac{4(1-a_0)|q|}{1-2|q|}.
\end{aligned}
\end{equation*}
An easy computation shows that 
$$a_0 + \frac{4(1-a_0)|q|}{1-2|q|}< 1 \quad \text{if and only if} \quad  |q| \leq \frac{1}{6}.$$
Then we can conclude that $\sum_{n\geq 0}|q^n a_n|<1$ for all $|q| \leq \frac{1}{6}$.

\end{proof}

\section{The Bohr Theorem}
We will prove now the quaternionic analog of the sharp version of the Bohr Theorem. 
\begin{teo}[Bohr]
Let $f(q)=\sum_{n\geq 0}q^n a_n$ be a regular function on $\B$, continuous on the closure $\overline{\B}$, such that $|f(q)|<1$ for all $|q|\leq 1$. Then
\begin{equation*}
\sum_{n\geq 0}|q^n a_n|<1
\end{equation*}
for all $|q| \leq \frac{1}{3}$.
Moreover $\frac{1}{3}$ is the largest radius for which the statement is true.
\end{teo}
\begin{proof}
As in the previous case, we can suppose $0\leq a_0<1$.
We want to improve the estimate of $|a_n|$, showing that $|a_n|<1-a_0^2$ for all $n\geq 1$.
Let us first treat the case $n=1$.
Consider the function defined as
$$H(q)=(1-f(q)a_0)^{-*}*(q^{-1}(f(q)-a_0)).$$   
Since $f: \B \to \HH$ is continuous up to the boundary, $f(0)=a_0<1$ and $|f(q)|<1$ on $\overline\B$, we get that $H$ is regular on $\B$ and continuous up to the boundary.
Moreover, if we set $T(q)=((1-f(q)a_0)^c)^{-1}q(1-f(q)a_0)^c$, by Proposition \ref{Caterina}, we can write
\begin{equation*}
\begin{aligned}
H(q)&=(1-f(q)a_0)^{-*}*(q^{-1}\sum_{n\geq 1}q^n a_n)=(1-f(q)a_0)^{-*}*\sum_{n\geq 1}q^{n-1} a_n\\
&=(1-f(T(q))a_0)^{-1}\sum_{n\geq 1}T(q)^{n-1} a_n.
\end{aligned}
\end{equation*}
Furthermore, $T(0)=0$ implies that $H(0)=(1-a_0^2)^{-1}a_1$, and, since $|T(q)|=|q|$, by the Maximum Modulus Principle \ref{PMM} we get that
\begin{equation}
\begin{aligned}\label{numero}
|H(q)|&\leq \max_{|q|=1}|H(q)|\\
&=\max_{|q|=1}|1-f(T(q))a_0|^{-1}|T(q)|^{-1}|f(T(q))-a_0|\\
&= \max_{|q|=1}|1-f(T(q))a_0|^{-1}|f(T(q))-a_0|.
\end{aligned}
\end{equation}
Notice that $(1-qa_0)^{-1}(q-a_0)$ is a fractional linear transformation that maps $\B$ into $\B$. Hence, since $|f(T(q))|<1$, by \eqref{numero} we get that $|H(q)|<1$ for all $q \in \B$. In particular we have then $|H(0)|<1$, that is
$$a_1<1-a_0^2.$$ 
For the case $n>1$ we want to build a function with the same properties as $f$, whose coefficient of the first degree term is $a_n$. Therefore, with the same argument used for $n=1$, we would obtain that $|a_n|< 1-a_0^2.$
Let $\omega$ be a quaternionic primitive $n$-th root of unity (see \cite{GSV}) and let $I$ be such that $\omega \in L_I$.
Consider the function defined on $L_I\cap \B=\B_I$ as
$$F_I(z)=f_I(z)+f_I(z\omega)+\dots+f_I(z\omega^{n-1})$$
where $f_I$ is the restriction of $f$ to the plane $L_I$. The function 
$F_I$ is holomorphic on $\B_I$, continuous on $\overline{\B_I}$, and its modulus
$$|F_I(z)|\leq |f_I(z)|+|f_I(z\omega)|+\dots+|f_I(z\omega^{n-1})|<n.$$
Its power series expansion is
\begin{equation*}
\begin{aligned}
F_I(z)&=\sum_{m \geq 0}z^m a_m +\sum_{m \geq 0}z^m\omega^m a_m+ \dots  + \sum_{m \geq 0}z^m\omega^{(n-1)m} a_m\\ 
&=\sum_{m \geq 0}z^m \sum_{k=0}^{n-1}\omega^{km}a_m.
\end{aligned}
\end{equation*}
Now if $n \mid m$, then $\omega^{m}=1$ and
$$\sum_{k=0}^{n-1}\omega^{km}=n.$$
Otherwise $\omega^{m}$ is a $n$-th root of unity different from $1$, and hence a root of the polynomial $z^{n-1}+\dots+z+1$,  that yields
$$\sum_{k=0}^{n-1}\omega^{km}=0.$$
Therefore 
$$F_I(z)=\sum_{m \geq 0}z^{nm} n a_{nm}.$$
Notice that the coefficients  of $F_I$ do not depend on $\omega$, and hence neither on $I$. This implies that
for all $J \in \s$, if $\omega_J$ is an $n$-th root of unity in $L_J$, the function $F_J(u)$ defined on $\B_J$ as
$$F_J(u)=f_J(u)+f_J(u\omega_J)+\dots+f_J(u\omega_J^{n-1})$$
has the same regular extension (see Lemma \ref{extensionlemma}) to the unit ball $\B$ as $F_I$, namely
$$F(q)=\sum_{m\geq 0}q^{nm}n a_{nm}.$$ 
\noindent Hence we have that $|F(q)|<n$ for all $q \in \B$. 
If now we set $w=q^m$ and we consider the function $\Phi=\frac{F}{n}$,
$$\Phi(w)=\sum_{m\geq 0}w^m a_{mn}=a_0 + wa_n+w^2 a_{2n}+\dots \ ,$$
then we obtain that $\Phi$ is regular on $\B$, continuous on $\overline{\B}$ and $|\Phi(w)|<1$ for all $q\in \B$.
Since the coefficient of the first degree term of $\Phi$ is $a_n$, with the same argument used to prove that $|a_1|<1 -a_0^2$, we can conclude that $|a_n|<1-a_0^2$.
Since $0\leq a_0<1$, we get that 
$$|a_n|<(1+a_0)(1-a_0)<2(1-a_0).$$
Therefore, for $|q|=\frac{1}{3}$,
\begin{equation*}
\begin{aligned}
\sum_{n\geq 0}|q^n a_n|&=\sum_{n\geq 0}\frac{1}{3^n} |a_n|<a_0 + 2(1-a_0)\sum_{n\geq 1}\frac{1}{3^n}\\
&=a_0 +2(1-a_0)\frac{1}{2}=a_0+1-a_0=1.
\end{aligned}
\end{equation*}

\noindent To show that the statement does not hold for any radius larger than $\frac{1}{3}$, we will proceed as follows. For any point $q_0\in \B$ such that $|q_0|>\frac13$, we will find a regular function $g_{q_0}:\B\to \B$  continuous up to the boundary, such that $g_{q_0}(q)=\sum_{n \geq 0}q^nb_n$ and $\sum_{n \geq 0}|q_0^nb_n|>1$.

\noindent To start with, take $a \in (0,1)$ and consider the function
$$\varphi(q)=(1-qa)^{-*}*(1-q)=(1-qa)^{-1}(1-q).$$
Since $a<1$, then $\varphi$ is regular on $\B$, continuous on $\overline{\B}$ and since it has real coefficients, $\varphi$ is slice preserving.
By the Maximum Modulus Principle \ref{PMM}
$$\max_{q \leq 1}|\varphi(q)|=\max_{|q|=1}|\varphi(q)|,$$
and for all $I \in \s$, 
$$\max_{|q|=1}|\varphi(q)|=\max_{|z|=1}|\varphi_I(z)|=\frac{2}{1+a}>1.$$
If the power series expansion of $\varphi$ is $\sum_{n \geq 0}q^nb_n$, with $b_n\in \mathbb{R}$, then
\begin{equation*}
\sum_{n \geq 0}q^nb_n=(1-qa)^{-1}(1-q)=(\sum_{n \geq 0}q^na^n)(1-q)=1 +q(a-1)\sum_{n \geq 0}q^na^n
\end{equation*}
and hence
\begin{equation*}
\sum_{n \geq 0}|q^nb_n|=1 +|q|(1-a)\sum_{n \geq 0}|q|^na^n=1+\frac{|q|(1-a)}{1-|q|a}.
\end{equation*}
In particular 
\begin{equation*}
\sum_{n \geq 0}|q^nb_n|>\frac{2}{1+a} \quad \text{if and only if}\quad 1+\frac{|q|(1-a)}{1-|q|a}> \frac{2}{1+a},
\end{equation*}
that holds if and only if
$$|q|>\frac{1}{1+2a}.$$

\noindent Fix now $q_0 \in \B$ such that $|q_0|>\frac{1}{3}$. Then we can take $a\in(0,1)$ such that $|q_0|>\frac{1}{1+2a}>\frac{1}{3}$. Therefore the correspondent $\varphi$ is such that
$$\sum_{n\geq 0}|q_0|^n|b_n|> \frac{2}{1+a}$$
Let us consider the function $\varphi_c$, defined as
$$\varphi_c(q)=c\varphi(q)=c(1-qa)^{-1}(1-q),$$
where $c\in(0,1)$. Then $\varphi_c$ is regular on $\B$, continuous on $\overline{\B}$, and its maximum modulus is
$$\max_{|q|=1}|\varphi_c(q)|=\frac{2c}{1+a}.$$
Moreover, its power series expansion is obtained multiplying by $c$ the one of $\varphi$, so with the same calculation that we have done for $\varphi$, we get that
$$c\sum_{n\geq 0}|q^nb_n|> \frac{2c}{1+a}$$
if and only if
$$|q|>\frac{1}{1+2a}.$$
Hence 
$$c\sum_{n\geq 0}|q_0^nb_n|> \frac{2c}{1+a}.$$
To conclude, notice that we can choose $ c\in (0,1)$ such that
$$c\sum_{n\geq 0}|q_0^nb_n|>1> \frac{2c}{1+a}$$
and set $$g_{q_0}=\varphi_c.$$
\end{proof}

\end{document}